\documentclass[a4paper,11pt]{article}
\usepackage[latin1]{inputenc}
\usepackage[brazil]{babel}
\usepackage{amssymb}
\usepackage{amsthm, amsfonts}
\usepackage{amsmath}
\usepackage{float}%Caso deseje posicionar o objeto no ponto exato 
\usepackage{graphicx} %usar esse pacote quando for compilar pra dvi 

\newtheorem{teo}{Teorema}[section]
\newtheorem{cor}[teo]{Corolário}
\newtheorem{lem}[teo]{Lema}

\newtheorem{defi}{Definição}[section]

\newcommand{\R}{\mathbb{R}}

\newcommand{\N}{\mathbb{N}}
\newcommand{\E}{\mathbb{E}}
\newcommand{\F}{\mathbb{F}}

\newcommand{\pid}{\rangle}
\newcommand{\pie}{\langle}

\title{Extensão de aplicações na esfera de um espaço vetorial com produto interno}
\author{Jose Edson Sampaio\\
		Instituto de Federal de Educação, Ciência e Tecnologia do Ceará}
\makeindex

\begin{document}
\maketitle

\begin{abstract}
Being $\E$ a vector space with inner product and $\mathbb{S}_{\mathbb{\E}}$ the sphere of $\E$, will be given a demonstration that every application of the sphere $\mathbb{S}_{\mathbb{\E}}$ itself it such that preserve inner product is the restriction of a linear isometry in $\E$.
\end{abstract}

\section{Introdução\label{sec_1}}
Neste texto apresentaremos uma demonstração diferente da feita em \cite{YW} e \cite{DT} de que uma isometria na esfera de um espaço vetorial com produto interno se extende a uma isometria linear em todo espaço, desde que a isometria na esfera preserve produto interno, para espaços de dimenção infinita e finita respectivamente. Em um certo sentido, isso generaliza o teorema 2 de \cite{YW} que prova esse fato para espaços $l^p$ com $p>1$. Para isso, apresentaremos algumas definições antes dos resultados.
\begin{defi}
Sejam $(\E,\|.\|_1)$ e $(\F,\|.\|_2)$ dois espaços vetoriais normados e uma aplicação linear $F:\E\longrightarrow \F$. Dizemos que $F$ é uma isometria se $\|F(x)-F(y)\|_2=\|x-y\|_1$ para quaisquer $x,y\in \E$.
\end{defi}
\begin{defi}
Sejam $(\E,\pie,\pid_1)$ e $(\F,\pie,\pid_2)$ dois espaços vetoriais com produto interno e uma aplicação $F:A\longrightarrow B$, onde $A\subset\E$ e $B\subset\F$. Dizemos que $F$ preserva produto interno se $\pie F(x),F(y)\pid_2=\pie x,y\pid_1$ para quaisquer $x,y\in A$. 
\end{defi}
\begin{defi}
Dado $(\E,\|.\|)$ um espaço vetorial normado, a esfera de $\E$ será denotada por $\mathbb{S}_{\mathbb{E}}=\{x\in \E; \|x\|=1\}$. 
\end{defi}

\begin{lem}\label{pres_angulo}
Se $\varphi: \mathbb{S}^n\longrightarrow \mathbb{S}^n$ diferenciável é tal que a diferencial $d\varphi$ preserva produto interno então a própria $\varphi$ preserva produto interno.
\end{lem} 
\begin{proof}
Pelo teorema 4.3 do capitulo 6, do livro Elementary differential geometry \cite{BO} temos que $\varphi$ preserva distância intrinseca.
E ainda no mesmo livro, no exemplo 1.9(b) do capitulo 8, temos que a distância intrinseca de dois pontos $p,q\in\mathbb{S}^n$ é justamente o ângulo entre $p$ e $q$ já que estamos com a esfera unitária. E como $\varphi$ preserva distância intrinseca, temos portanto que $\varphi$ preserva ângulos e assim preserva produto interno.
\end{proof}

\begin{lem}\label{pres_dist_1}
Seja $(\E,\pie,\pid)$ um espaço vetorial com produto interno. Se $F:\E\longrightarrow \E$ preserva distância e $F(0)=0$ então $F$ é uma isometria. 
\end{lem} 
\begin{proof}
Façamos primeiramente para o caso em que $dim \E = \infty $.
Como $F$ preserva distância, temos que $\|F(x)-F(y)\|=\|x-y\|$ para quaisquer $x,y\in \E$ e como $F(0)=0$ temos que $\|F(x)\|=\|x\|$ para todo $x\in \E$. Mas para todo $x,y\in \E$ temos
$$\|F(x)-F(y)\|^2=\|F(x)\|^2 - 2\pie F(x),F(y)\pid + \|F(y)\|^2 $$
e
$$\|x-y\|^2=\|x\|^2 - 2\pie  x , y\pid + \|y\|^2.$$
E então, $\pie F(x),F(y)\pid = \pie  x , y\pid$ quaisquer que sejam $x,y\in \E$.
Sejam agora, $\{e_j\}_{j=1}^\infty \subset\E$ uma base ortonormal, $v=(v_j)_{j=1}^\infty$ e $w=(w_j)_{j=1}^\infty $ vetores quaisquer em $\E$. Daí, sendo $\lambda \in \R$ temos,
$$F(v+\lambda w)= F(\sum\limits_{j=1}^\infty v_j+\lambda w_j)e_j).$$
E como $F$ preserva produto interno temos que $\{F(e_j)\}_{j=1}^\infty \subset\E$ também é base ortonormal e assim,
$$F(v+\lambda w)= \sum\limits_{j=1}^\infty \pie F(v+\lambda w), F(e_j)\pid F(e_j).$$
E então para mostrar que $F$ é linear e portanto isometria (pois $F$ preserva distância), é suficiente mostrar que $\pie F(v+\lambda w), F(e_j)\pid = v_j+\lambda w_j$ para todo$j\in\N$. Mas isso é justamente o que ocorre, já que $\pie F(v+\lambda w), F(e_j)\pid = \pie v+\lambda w, e_j\pid = v_j+\lambda w_j$ para todo $j\in\N$.
E para o caso em que $dim \E = n+1 < \infty $ e neste caso $\E\cong \R^{n+1}$ basta tomar a base ortonormal como sendo a base canônica do $\R^{n+1}$ e o resto é análogo ao feito no caso acima.
\end{proof}

\section{Resultado principal\label{sec_2}}
\begin{teo}\label{teo_principal}
Seja $(\E,\pie,\pid)$ um espaço vetorial com produto interno. Se $\varphi: \mathbb{S}_{\mathbb{E}}\longrightarrow \mathbb{S}_{\mathbb{E}}$ preserva produto interno então existe uma isometria $F:\E\longrightarrow \E$ tal que $\varphi = F|_{\mathbb{S}_{\mathbb{E}}}$. 
\end{teo} 
\begin{proof}
Defina $F:\E\longrightarrow \E$ por
$$
F(x)=\left\{\begin{array}{ll}
\| x \| \varphi(\frac{x}{\| x \|}),&x\not=0\\
0, &x=0
\end{array}\right.
$$
É claro que $\|F(x)\| = \|x\|$ para todo $x\in \E$. E para $x$ e $y$ diferentes de zero em $\E$ temos que,
\begin{eqnarray}
\|F(x)-F(y)\|^2&=&\|F(x)\|^2 - 2\pie F(x),F(y)\pid + \|F(y)\|^2 \nonumber\\ 
			   &=&\|x\|^2 - 2\|x\|\|y\|\pie \varphi(\frac{x}{\| x \|}), \varphi(\frac{y}{\| y \|})\pid + \|y\|^2 \label{eq1}
\end{eqnarray}
E como $\varphi$ preserva produto interno, temos que,
\begin{equation}
\pie \varphi(\frac{x}{\| x \|}), \varphi(\frac{y}{\| y \|})\pid = \pie \frac{x}{\| x \|}, \frac{y}{\| y \|}\pid \label{eq2}
\end{equation}
E assim pelas igualdades (\ref{eq1}) e (\ref{eq2}) temos,
\begin{equation}
\|F(x)-F(y)\|^2=\|x-y\|^2
\end{equation}
Logo, F preserva distância e então pelo lema (\ref{pres_dist_1}) $F$ é isometria.
\end{proof}

\begin{cor}
Se $\varphi: \mathbb{S}^n\longrightarrow \mathbb{S}^n$ diferenciável é tal que a diferencial $d\varphi$ preserva produto interno então existe uma $F\in O(n+1)$ tal que $\varphi = F|_{\mathbb{S}^n}$.
\end{cor}
\begin{proof}
Decorre direto do lema (\ref{pres_angulo}) e do teorema (\ref{teo_principal}).
\end{proof}
%Bibliografia
\nocite{YW, DT, BO, HB, BPR} %Lista na bibliografia os item não citados no texto com o comando \cite{}, seprados por virgula
%\addcontentsline{toc}{section}{Sumário} %Comando para que a bibliografia apareça no sumário
\bibliography{biblio_isometry_in_sphere}
\end{document}